\documentclass{amsart}
\usepackage{amssymb}

\usepackage{graphicx}
\usepackage{epic,eepic}

\newtheorem{theorem}{Theorem}[section]
\newtheorem{lemma}[theorem]{Lemma}
\newtheorem{proposition}[theorem]{Proposition}
\newtheorem{corollary}[theorem]{Corollary}

\theoremstyle{definition}
\newtheorem{definition}[theorem]{Definition}

\theoremstyle{remark}
\newtheorem{remark}[theorem]{Remark}

\numberwithin{equation}{section}

\newcommand{\BB}[1]{\mathbb{#1}}

\newcommand{\BS}[1]{\boldsymbol{#1}}

\newcommand{\COMP}{\raisebox{0.3ex}{\hspace{0.3ex}%
   {$\scriptstyle{\circ}$}\hspace{0.5ex}}}   %%% composition

\newcommand{\HEAD}[3]{\put(#1,#2)%
   {\makebox(0,0){\rotatebox{#3}%
   {\raisebox{.5ex}{$\blacktriangleright\phantom{i}$}}}}}
\newcommand{\WHEAD}[3]{\put(#1,#2)%
   {\makebox(0,0){\rotatebox{#3}%
   {\raisebox{.5ex}{$\rhd\phantom{i}$}}}}}

\newcommand{\ORA}[1]{\overrightarrow{#1}}
\newcommand{\PLL}{\mathsf{PLL}}

\newcommand{\UL}{\underline}

\begin{document}

\title{Sufficiency of simplex inequalities}


%    author one information
\author{Shuzo Izumi}
\address{Research Center of Quantum Computing,  
Kinki University, Higashi-Osaka 577-8502, Japan}
\curraddr{}
\email{sizmsizm@gmail.com}
\thanks{}

%    \subjclass is required.
\subjclass[2010]{Primary 51M16}

\date{\today}

\dedicatory{}

%    "Communicated by" -- provide editor's name; required.
\commby{Daniel Ruberman}

%    Abstract is required.
\begin{abstract}
Let $z_0,\dots,z_n$ be the $(n-1)$-dimensional volumes of facets 
of an $n$-simplex. Then we have the simplex inequalities: 
$z_p < z_0+\dots+\check{z}_p+\dots+z_n$ $(0\le p\le n)$, 
generalizations of the triangle inequalities. 
Conversely, suppose that numbers 
$z_0,\dots,z_n>0$ satisfy these inequalities. 
Does there exist an $n$-simplex the volumes of whose facets 
are them? Kakeya solved this problem affirmatively 
in the case $n=3$ 
and conjectured the assertion for all $n\ge 4$.  
We prove that his conjecture is affirmative. 
\end{abstract}

\maketitle

\section{Introduction.}
%%%%%%%%%%%%%%%%%%%%%%%%%%%%%%%%%%%%
Let $V\subset\BB{R}^n$ be an $n$-simplex. 
We assume that $n\ge 2$ for $n$-simplices 
throughout this paper. 
Let us call a 
vector \emph{inward facet normal} 
if it is orthogonal to a facet of $V$, oriented 
inward and with length equal to the $(n-1)$-dimensional 
volume of the facet. (In reality, they are multiplied by $n-1$ 
in this paper for the sake of notational simplicity.) 
We call such a volume \emph{facet volume}. 
A simplex is said to be 
\emph{non-degenerate} if it is not included in a hyperplane. 
The following facts are known about the facet normals of an 
$n$-simplex. 
\begin{description}
\item[(a)]  
Let $\BS{z}_0,\dots,\BS{z}_n$ be inward facet normals of 
a non-degenerate $n$-simplex 
$V\subset\BB{R}^n$. Then $\BS{z}_0,\dots,\BS{z}_n$ span 
$\BB{R}^n$ and $\BS{z}_0+\cdots+\BS{z}_n=\BS{0}$. 
Physically, this equality corresponds to the fact that, 
if we neglect variation of pressure with respect to depth, 
the pressure put on an object in static liquid is 
in equilibrium. 
\item[(b)] (A special case of Minkowski's existence theorem.) 
If vectors $\BS{z}_0,\dots,\BS{z}_n\in\BB{R}^n$ span $\BB{R}^n$ 
and satisfy 
$\BS{z}_0+\cdots+\BS{z}_n=\BS{0}$, then there exists a 
non-degenerate $n$-simplex whose inward facet normals are 
$\BS{z}_0,\dots,\BS{z}_n$. Furthermore, the simplex is 
determined by $\BS{z}_0,\dots,\BS{z}_n$ up to translation. 
(Minkowski has proved a similar result 
for general convex polytopes. See e.g. \cite{schneider}.) 
\item[(c)] 
%(The simplex inequalities.) 
Suppose that $\BS{z}_0,\dots,\BS{z}_n\in\BB{R}^n$ 
span $\BB{R}^n$ 
and that $\BS{z}_0+\cdots+\BS{z}_n=\BS{0}$. 
Then we have 
\begin{align}
2|\BS{z}_k| < |\BS{z}_0|+\cdots+|\BS{z}_n|\quad(0\le k\le n)
\tag{$*$}
.\end{align}
\end{description}
\vskip1.5ex

By (a)$+$(c), we see that the inequalities $(*)$ are 
necessary condition for facet volumes of a 
non-degenerate $n$-simplex. Hence they are called 
\emph{simplex inequalities}. 
In this note we prove that these inequalities 
are the only 
condition for the $n+1$ positive numbers to be the set of 
facet volumes of a non-degenerate $n$-simplex 
$V\subset\BB{R}^n$. The case $n=2$ is very easy. The case 
$n=3$ is proved by S. Kakeya \cite{kakeya}. 
He added ``The higher dimensional case may be proved 
in a similar way". Probably, this may be conjectured also by 
many mathematicians and non-mathematicians. 

In view of Minkowski's theorem (b), 
we have only to prove the following. 
\vskip1.5ex

%%%%%%%%%%%%%%%%%%%%%
\begin{description}
\item[(d)]
If positive numbers $z_0,\dots,z_n$ satisfy 
$2z_k < z_0+\cdots+z_n\quad(0\le k\le n)$, 
then there exist $n$ vectors 
$\BS{z}_0,\dots,\BS{z}_n\in\BB{R}^n$ which span $\BB{R}^n$ 
and satisfy 
\begin{gather*}
\BS{z}_0+\cdots+\BS{z}_n=\BS{0},\quad
|\BS{z}_0|=z_0,\dots,|\BS{z}_n|=z_n 
.\end{gather*} 
\end{description}
%%%%%%%%%%%%%%%%%%%%%
%In the case $n=3$, assertion $(2\Longrightarrow1)$ is 
%geometrically obvious from the observation of shadows 
%of facets to 
%a hyperplane which is parallel to a facet. The general case 
%follows similarly (see D. Klain \cite{klain}). 
Our proof is quite different from Kakeya's proof 
for the case $n=3$. 

In order to prove this, we associate three spaces of ordered 
set of vectors to a given simplex: 
the space of position vectors of vertices, 
that of edges vectors, 
and that of facet normals. 
We call these spaces \emph{loop spaces}. 
Our proof is described by
the properties of maps among these loop spaces. 
That will seem overdoing but it will be helpful 
to avoid confusion 
when we identify one loop space with another 
(see the last paragraph of \S \ref{repetition}).  
The problem (d) has an infinite numbers of solutions, 
which is a troublesome point as usual in such an 
existence theorem. 
The only prerequisite for this paper is the elementary 
linear algebra. 
%%%%%%%%%%%%%%%%%%%%%%%%%%%%%%%%%%%%
\section{Cofactor matrices.}
%%%%%%%%%%%%%%%%%%%%%%%%%%%%%%%%%%%%%
Let us recall some properties of the cofactor matrix 
of an $n\times n$ matrix. 
It is convenient to introduce the \emph{vector product} 
$[\BS{w}_1,\dots, \BS{w}_{n-1}]$ of vectors 
$\BS{w}_1,\dots, \BS{w}_{n-1}$ in $\BB{R}^n$ 
as a generalisation of the usual vector product 
of two vectors in $\BB{R}^3$. It is defined by the following. 
%%%%%%%%%%%%%%%%%%
\begin{proposition}
If an exterior product $\BS{w}_1\wedge\dots\wedge\BS{w}_{n-1}$ 
of vectors in $\BB{R}^n$ is given, there exists a unique vector 
$[\BS{w}_1,\dots, \BS{w}_{n-1}]\in\BB{R}^n$ such that, 
for any $\BS{w}_0\in\BB{R}^n$, 
$$
(\BS{w}_0\cdot[\BS{w}_1,\dots, \BS{w}_{n-1}])
\BS{e}_1\wedge\cdots\wedge\BS{e}_{n}
=\BS{w}_0\wedge\dots\wedge\BS{w}_{n-1}
%\BS{w}_0\wedge\dots\wedge\BS{w}_{n-1}
,$$
where $\BS{e}_1,\dots,\BS{e}_n$ 
denote the unit coordinate vectors.  
\end{proposition}
%%%%%%%%%%%%%%
The product $[\BS{w}_1,\dots, \BS{w}_{n-1}]$ is given 
as follows. Take the $(n\times n)$ 
matrix $A:=(\BS{w}_0,\dots, \BS{w}_{n-1})$ whose $j$-th column 
vector is $\BS{w}_j$. Then the first row vector of the 
cofactor matrix $c(A)$ of $A$ has the property required for 
the product. We skip the proof. We can also easily verify 
the following properties of this product. 
%%%%%%%%%%%%%%%
\begin{lemma}\label{vector-product}
\begin{enumerate}
\item
If $A:=(\BS{w}_1,\dots, \BS{w}_{n})\in(\BB{R}^n)^{n}$, we have 
\begin{gather*}
c(A)=\left([\BS{w}_2,\dots, \BS{w}_{n}],-[\BS{w}_1,\BS{w}_3,
\dots, \BS{w}_{n}],
%\\
\dots,(-1)^{n+1}[\BS{w}_1,\dots, \BS{w}_{n-1}]\right)
.\end{gather*}
\item
This product defines an alternating $(n-1)$-linear map 
$(\BB{R}^n)^{n-1}\longrightarrow\BB{R}^n$. 
\item
The vector $[\BS{w}_1,\dots, \BS{w}_{n-1}]$ is orthogonal to 
$\BS{w}_1,\dots, \BS{w}_{n-1}$. It has length equal to the 
$(n-1)$-dimensional volume of the parallelotope 
with generating edges $\BS{w}_1,\dots, \BS{w}_{n-1}$. 
\item 
We have an equality for the scalar product: 
$$
(\BS{w}_0\cdot[\BS{w}_1,\dots, \BS{w}_{n-1}])
=\det(\BS{w}_0,\dots, \BS{w}_{n-1})
.$$
\item
We have
$$
\rm{det}([\BS{w}_1,\dots, \BS{w}_{n-1}],\BS{w}_1,
\dots,\BS{w}_{n-1})
=|[\BS{w}_1,\dots, \BS{w}_{n-1}]|^2\ge 0
$$ 
and the equality holds if and only if 
$\BS{w}_1,\dots, \BS{w}_{n-1}$ are linearly dependent. 
\end{enumerate}
\end{lemma}
%%%%%%%%%%%%%%
\begin{lemma}\label{cramer}
Let $M,\ N$ be $(n\times n)$-matrices $(n\ge 2)$. 
Then $M$ is invertible if and only if $c(M)$ is so 
and we have the following: 
\begin{gather*}
c(M)M^t=M^tc(M)=(\det M)\cdot E,\quad
\det(c(M))=(\det(M))^{n-1},\quad
\\
c(\lambda\cdot M)=\lambda^{n-1}\cdot c(M),\quad
c(MN)=c(M)c(N),
\\
c\COMP c(M)=(\det M)^{n-2}\cdot M.
\end{gather*}
\end{lemma}
%%%%%%%%%%%%%%%%%%%%
The first formula is Cramer's, 
where $M^t$ indicates the transposed matrix of $M$.  
Others follow from it in a standard way. 
%%%%%%%%%%%%%%%%%%%%%%%%%%%%%%%%%%
\section{Loop spaces.}\label{loops-paces}
%%%%%%%%%%%%%%%%%%%%%%%%%%%%%%%%%%%%%%%
Let us associate three spaces of $n$-loops to each simplex. 
%%%%%%%%%%%%%%%%%%%%%%%%
\begin{definition}\label{PLL}
Let us put 
$$
\PLL^n:=\{V:=(\BS{v}_0,\dots,\BS{v}_{n})\in(\BB{R}^n)^{n+1}:
\BS{v}_0+\dots+\BS{v}_{n}=\BS{0}\}
$$
and call its element $V\in\PLL^n$ an \emph{$n$-loop}, 
because we can associate it a Piecewise Linear Loop 
${\rm P}_0{\rm P}_1\cdots{\rm P}_n{\rm P}_0$ 
with base point ${\rm P}_{n+1}={\rm P}_{0}$ such that 
$\BS{v}_p=\ORA{{\rm P}_{p-1}{\rm P}}_{p}$. The barycentre 
$(\BS{v}_0+\dots+\BS{v}_{n})/(n+1)$ of these points is 
the origin O. The \emph{main part} $\UL{V}$ of $V\in\PLL^n$ 
is defined by $\UL{V}:=(\BS{v}_1,\dots,\BS{v}_{n})$. 
It is often considered as the $n\times n$ matrix whose 
$p$-th column vector is $\BS{v}_p$. 
We call an element $V\in\PLL^n$ \emph{affine independent} 
(resp. \emph{positive}) if $\det\UL{V}\neq0$ 
(resp. $\det\UL{V}>0$). 
The set of positive $n$-loops is denoted by $\PLL_+^n$.  
\end{definition}
%%%%%%%%%%%%%%%%%%%%%%%%
\begin{remark}
Affine independence of $V$ may be defined by linear independence of  
$\BS{v}_1-\BS{v}_0,\dots,\BS{v}_n-\BS{v}_0$. 
Using the assumption $V\in\PLL$, we see that this is equivalent 
to linear independence of $\{\BS{v}_1,\dots,\BS{v}_n\}$ and to the 
condition that $\det\UL{V}\neq0$ above. Similarly affine independence 
of $V$ is equivalent to linear independence of any other $n$ 
vectors of $\{\BS{v}_0,\dots,\BS{v}_n\}$. 
\end{remark}
%%%%%%%%%%%%%%%%
\begin{definition}\label{labelled}
Take a labelled $n$-dimensional simplex 
with the barycentre at ${\rm O}$. That is, 
its vertices are named as ${\rm P}_0,\dots,{\rm P}_n$ 
and the position vectors $\BS{v}_0,\dots,\BS{v}_n$ of 
respective vertices satisfy $\BS{v}_1+\cdots+\BS{v}_n=\BS{0}$. 
The simplex can be 
expressed by $V=(\BS{v}_0,\dots,\BS{v}_n)\in\PLL^n$, which 
we call the \emph{vertex loop} of the symplex.  
A simplex $V$ is degenerate if it is contained in a 
hyperplane. This is equivalent to say that it is not affine 
independent as an $n$-loop.  
\end{definition}
%%%%%%%%%%%%%%%
\begin{definition}\label{edge} 
Take the set of the edge vectors 
\begin{gather*}
\BS{w}_{0}:=\ORA{{\rm P}_{n}{\rm P}}_{0}=\BS{v}_0-\BS{v}_n,\ 
\BS{w}_{1}:=\ORA{{\rm P}_{0}{\rm P}}_{1}=\BS{v}_1-\BS{v}_0,\ 
\\
\dots,\ 
\BS{w}_{n}
:=\ORA{{\rm P}_{n-1}{\rm P}}_{n}=\BS{v}_{n}-\BS{v}_{n-1}
\end{gather*}
of a simplex $V=(\BS{v}_0,\dots,\BS{v}_n)\in\PLL^n$.  
This does not contain all the edges of $V$. It forms 
an $n$-loop in $\PLL^n$. We define the \emph{edge map}: 
\begin{gather*}
\varphi:\PLL^n\longrightarrow\PLL^n,
\\
V:=(\BS{v}_0,\dots,\BS{v}_n)\longmapsto
\BS{w}:=(\BS{w}_{0},\BS{w}_{1},\dots,\BS{w}_{n})
.\end{gather*}
Since 
$$
\BS{w}_i=\BS{v}_1+\cdots+\BS{v}_{i-1}
+2\BS{v}_i+\BS{v}_{i-1}+\cdots+\BS{v}_n
\quad (1\le i\le n)
,$$
computing the matrix representing $\varphi$, 
we see that $\varphi$ is bijective. 
\end{definition}
%%%%%%%%%%%%%%%%%%%%%%%%%%%%%%
%\section{Space of facet loops.}
%%%%%%%%%%%%%%%%%%%%%%%%%%%%%%%
We define the facet normal $\BS{z}_p$ of the facet opposite to 
${\rm P}_p$ (or $\BS{v}_p$) $(p\neq0)$ of an $n$-simplex $V$ 
as the vector orthogonal to the edges 
$$
\BS{w}_0,\BS{w}_1,\dots,\check{\BS{w}}_p,
\check{\BS{w}}_{p+1},\dots,\BS{w}_{n}
\quad
(\check{\phantom{w}}\text{: omission})
$$
on the boundary of the facet 
and with length equal to the $(n-1)$-dimensional volume of the 
parallelotope generated by these edge vectors. 
(It may be conventional to define the vector length 
of the facet normal as the $(n-1)$-dimensional volume 
$|\BS{z}_p|/(n-1)$ of the \emph{facet} but we delete 
the denominator $n-1$ for the sake of simplicity.) 
Thus we define as follows. 
%%%%%%%%%%%%%%%%%%%%%%%%%%%%%%%%%%%
\begin{definition}\label{facet-normal}
The \emph{facet normal opposite to} ${\rm P}_p$ 
(or $\BS{v}_p$) is 
\begin{gather*}
\BS{z}_p
:=
\begin{cases}
-[\BS{w}_0,\BS{w}_1,\dots,\check{\BS{w}}_p,
\check{\BS{w}}_{p+1},\dots,\BS{w}_n]
& (p\neq n)
\\
(-1)^{n+1}[\BS{w}_1,\dots,\BS{w}_{n-1}]
& (p=n)
\end{cases}
.\end{gather*}
\end{definition}
%%%%%%%%%%%%%%%%%%%%%
The meaning of the choice of the orientation of this normal is 
clarified at the top of \S \ref{positive}. 
Temporarily, we define the \textit{facet map} by
$$ 
\psi:\PLL^n\longrightarrow(\BS{R}^n)^{n+1},
\quad
(\BS{w}_{0},\dots,\BS{w}_{n})
\longmapsto(\BS{z}_0,\dots,\BS{z}_n) 
.$$
We shall shrink the target space after Proposition 
\ref{facetsformloop}. 
%%%%%%%%%%%%%%%%%%%%
\begin{proposition}\label{scalar}
If 
\begin{gather*}
V=(\BS{v}_0,\dots,\BS{v}_n)\in\PLL^n, 
\quad
W=(\BS{w}_0,\dots,\BS{w}_n):=\varphi(V)\in\PLL^n,
\\
Z=(\BS{z}_0,\dots,\BS{z}_n):=\psi(W)\in(\BS{R}^n)^{n+1}
,\end{gather*}
we have the following:
\begin{gather*}
\UL{Z}
=c(\BS{w}_1,\BS{w}_1+\BS{w}_2,\cdots,\BS{w}_1+\cdots+\BS{w}_n),
\\
(\BS{w}_p\cdot \BS{z}_p)=\det(\UL{W}),\quad
(\BS{w}_{p+1}\cdot \BS{z}_p)=-\det(\UL{W})
\\
(\BS{w}_p\cdot \BS{z}_q)=0\ (p\neq q,\ q+1),
\\
\det\UL{Z}
=(\det(\UL{W}))^{n-1},
\quad
\det(\UL{W})=(n+1)\det(\UL{V})
,\end{gather*}
where $(\phantom{b}\cdot\phantom{b})$ 
denotes the scalar product and $\BS{w}_{n+1}=\BS{w}_0$.  
\end{proposition}
%%%%%%%%%%%%%%%%%%%%%%%%%
%\renewcommand{\tablename}{\textsf{fig. 1}}
\unitlength=1mm
\begin{table}[h]
\begin{center}
\begin{picture}(110,35)
%%%%%%%%%%%%%%%%%%%%%%%%%%%%%%%%
\put(50,0){\scriptsize\textsf fig. 1}
\path(42,22)(32,32)(17,32)
\put(1.5,13){${\rm P}_n$}
\put(13,2){${\rm P}_0$}
\put(30,2){${\rm P}_1$}
\put(42,22.5){${\rm P}_{p-1}$}
\put(33,32){${\rm P}_p$}
\put(9,33){${\rm P}_{p+1}$}
\put(22,16){$z_p$}
%\path(15,5)(22,16.1)\path(24,19.3)(32,32)
\WHEAD{32}{32}{135}
%\put(17,17.2){\small$\BS{w}_1+\dots+\BS{w}_p$}
%%
\thicklines
\path(42,22)(17,32)
\path(5,15)(15,5)
\path(15,5)(30,5)\HEAD{30}{5}{0}
\put(21,3){\small$\BS{w}_1$}
\put(6,8){\small$\BS{w}_0$}
\put(35,6){\small$\BS{w}_2$}
\path(30,5)(42,12)\HEAD{42}{12}{33}
\path(42,12)(42,22)\HEAD{42}{22}{90}
%\put(43,15){\small$\BS{w}_{p-1}$}
%%
\put(37,27){\small$\BS{w}_p$}
\put(22,33.5){\small$\BS{w}_{p+1}$}
\HEAD{15}{5}{315}
\path(17,32)(5,25)\HEAD{5}{25}{210}
\path(5,25)(5,15)\HEAD{5}{15}{270}\put(0,20){\small$\BS{w}_n$}
\put(83,18){$z_n$}
%%%%%%%%%%%%%%%%%%%%%%%%%%%%%
\thinlines
\path(75,5)(90,5)(102,12)(102,22)(92,32)(77,32)%
(65,25)(65,15)(75,5)
\put(81,3){\small$\BS{w}_1$}
\put(61.5,13){${\rm P}_n$}
\put(73,2){${\rm P}_0$}
\put(88,2){${\rm P}_1$}
%\put(102,22.5){${\rm P}_{p-1}$}
%\put(93,32){${\rm P}_p$}
%\put(69,33){${\rm P}_{p+1}$}
\put(57,25){${\rm P}_{n-1}$}
\thicklines
\path(65,25)(75,5)
\path(75,5)(90,5)(102,12)(102,22)(92,32)(77,32)(65,25)%(65,15)
%\WHEAD{75}{5}{315}
\WHEAD{65}{15}{270}
\put(66,8){\small$\BS{w}_0$}
\HEAD{102}{12}{33}\put(95,6){\small$\BS{w}_2$}
\HEAD{90}{5}{0}\HEAD{102}{22}{90}
\HEAD{92}{32}{135}\HEAD{77}{32}{180}
\HEAD{65}{25}{210}%\HEAD{65}{15}{270}
\put(60,20){\small$\BS{w}_n$}
%%%%%%%%%%%%
\end{picture}
\end{center}
\end{table}
%%%%%%%%%%%%%%%%%%%%%%%%%%%%%%%
\begin{proof}
The first equality is obtained as
\begin{gather*}
\BS{z}_p=
-[\BS{w}_0,\dots,\check{\BS{w}}_p,\check{\BS{w}}_{p+1}
\dots,\BS{w}_n]
\\
=
[\BS{w}_1+\cdots+\BS{w}_n,\BS{w}_1,\dots,
\check{\BS{w}}_p,\check{\BS{w}}_{p+1},\dots,\BS{w}_n]
\\
=
[\BS{w}_p+\BS{w}_{p+1},\BS{w}_1,\dots,
\check{\BS{w}}_p,\check{\BS{w}}_{p+1},\dots,\BS{w}_n]
\\
=
(-1)^{p+1}[\BS{w}_1,\dots,\BS{w}_{p-1},\BS{w}_p+\BS{w}_{p+1},
\dots,\BS{w}_n]
\\
=
(-1)^{p+1}[\BS{w}_1,\BS{w}_1+\BS{w}_2,\cdots,
\check{\overbrace{\BS{w}_1+\cdots+\BS{w}_p}},
\dots,\BS{w}_1+\cdots+\BS{w}_n]
.\end{gather*}
The last equality is obtained as 
\begin{gather*}
\det(\UL{W})
=[\BS{w}_1,\BS{w}_1+\BS{w}_2,\cdots,
\BS{w}_1+\cdots+\BS{w}_n]
\\
=\det(\BS{v}_1-\BS{v}_0,\dots,\dots,\BS{v}_n-\BS{v}_0)
\\
=\det(\BS{v}_1+\dots+\BS{v}_n-n\BS{v}_0,
\BS{v}_2-\BS{v}_0,\dots,\BS{v}_n-\BS{v}_0)
\\
=\det(-(n+1)\BS{v}_0,\BS{v}_2-\BS{v}_0,\dots,\BS{v}_n-\BS{v}_0)
=-(n+1)\det(\BS{v}_0,\BS{v}_2,\dots,\BS{v}_n)
\\
=(n+1)\det(\BS{v}_1+\dots+\BS{v}_n,\BS{v}_2,\dots,\BS{v}_n)
=(n+1)\det(\UL{V})
.\end{gather*}
The others similarly follow from Proposition 
\ref{vector-product}, \ref{cramer}. 
\end{proof}
%%%%%%%%%%%%%%%%%%%%%%%%
\begin{proposition}\label{facetsformloop}
The set of facets of an $n$-simplex form a loop, 
that is $W:=(\BS{w}_0,\dots,\BS{w}_n)\in\PLL^n$ 
implies $\psi(W)\in\PLL^n$. 
\end{proposition}
%%%%%%%%%%%%%%%%%%%%%%%
\begin{proof}
Let us put $\psi(W)=(\BS{z}_0,\dots,\BS{z}_n)$. 
The sum $\BS{z}_0+\dots+\BS{z}_n$ is orthogonal to all 
$\BS{w}_0,\dots,\BS{w}_n$ by Proposition \ref{scalar}. 
If the simplex is non-degenerate, $\BS{w}_1,\dots,\BS{w}_n$ 
generate total space $\BB{R}^n$ by the last equality of 
Proposition \ref{scalar}. Then the sum is zero. 
The degenerate case follows from a continuity argument. 
\end{proof}
%%%%%%%%%%%%%%%%
By this proposition, we may modify the definition of the map 
$\psi:\PLL^n\longrightarrow(\BB{R}^n)^{n+1}$ 
to $\psi:\PLL^n\longrightarrow\PLL^n$.  
%%%%%%%%%%%%%%%%%%%%%%%%%%%%%%%
\section{Repetition implies similarity.}\label{repetition}
%%%%%%%%%%%%%%%%%%%%%%%%%%%%%%%%%%%%%
The cofactor matrix of the cofactor matrix of a 
square matrix is a constant multiple of the original one. 
This leads to an interesting and useful property of 
maps of loop spaces. 
%%%%%%%%%%%%
\begin{theorem}\label{iteration}
Let us take an $n\times n$ matrix $(n\ge 2)$  
$A:=
\begin{pmatrix}
2&1&\cdots&1\\
1&2&\cdots&1\\
\hdotsfor{4}\\
1&1&\cdots&2
\end{pmatrix}
$. 
Then we have 
$$
\UL{\psi}\COMP\UL{\varphi}(\UL{V})
=c(\UL{V}A),\quad
(\UL{\psi}\COMP\UL{\varphi})^2(\UL{V})
=(n+1)^{n-1}(\det\UL{V})^{n-2}\cdot\UL{V}
$$ 
for any $V\in\PLL^n$. Here, $\UL{\varphi}$ (resp. 
$\UL{\psi}$) denotes the restriction of $\varphi$ 
(resp. $\psi$) to main parts.  
\end{theorem}
%%%%%%%%%%%%%%%%%
\begin{proof}
Using the formulae in Lemma \ref{cramer}, \ref{scalar}, 
we can show that these follow from simple calculations: 
\begin{gather*}
\UL{\psi}\COMP\UL{\varphi}(\UL{V})
=c(\BS{w}_1,\BS{w}_1+\BS{w}_2,\cdots,\BS{w}_1+\cdots+\BS{w}_n)
\\
=c(\BS{v}_1-\BS{v}_0,\BS{v}_2-\BS{v}_0,\dots,\BS{v}_n-\BS{v}_0)
\\
=c(\BS{v}_1+(\BS{v}_1+\dots+\BS{v}_n),
\BS{v}_2+(\BS{v}_1+\dots+\BS{v}_n),
\\
\dots,
\BS{v}_n+(\BS{v}_1+\dots+\BS{v}_n)
)
=
c(\UL{V}A)
,\end{gather*}
\begin{gather*}
(\UL{\psi}\COMP\UL{\varphi})^2(\UL{V})
=c(c(\UL{V}A)A)=c(c(\UL{V}))c(c(A))c(A)
\\
=\{(\det\UL{V})^{n-2}\cdot\UL{V}\}\{c(c(A))(c(A))^t\}
\\
=(\det\UL{V})^{n-2}\cdot\UL{V}\{\det(c(A))\cdot E\}
=(\det\UL{V})^{n-2}\det(c(A))\cdot\UL{V}
\\
=(\det A)^{n-1}(\det\UL{V})^{n-2}\cdot\UL{V}
=(n+1)^{n-1}(\det\UL{V})^{n-2}\cdot\UL{V}
,\end{gather*}
where $E$ denotes the $n\times n$ identity matrix. 
\end{proof}
%%%%%%%%%%%%%%%
This theorem means the following. 
For any $n$-simplex $V\in\PLL^n$, 
consider its facet loop $Z:=\psi\COMP\varphi(V)\in\PLL^n$ 
as a new vertex loop and take its facet loop 
$Z':=(\psi\COMP\varphi)^2(V)\in\PLL^n$. 
If $Z'$ is considered as 
a new vertex loop again, then it is a simplex similar to $V$. 
This situation is somewhat analogous to the following problem 
of Langr on quadrangles \cite{langr}, which is already solved 
affirmatively:
\vskip1.5ex

\noindent\textit{
E1085. The perpendicular bisectors of the sides of 
a quadrilateral Q form a quadrilateral Q1, 
and the perpendicular 
bisectors of the sides of Q1 form a quadrilateral Q2.
Show that Q2 is similar to Q and find the ratio of similitude.}
\vskip1.5ex
\noindent
Refer to \cite{agaoka} for the history of this problem. 

%%%%%%%%%%%%%%%
\section{Spaces of positive loops}\label{positive}
%%%%%%%%%%%%%%%%%%%%%%%%%%%%%%
If $V\in\PLL_+^n$ (see Definition \ref{PLL}), 
the edge vector $\BS{w}_p$ and the facet normal $\BS{z}_p$ of  
$V$ satisfy $(\BS{w}_p\cdot \BS{z}_p)>0$ 
by Proposition \ref{scalar}. If the initial point of the vector 
$\BS{z}_p$ is put on the facet corresponding to 
${\rm P}_p$, its head is on the same side with 
${\rm P}_{p}$. This implies that $\BS{z}_p$ is 
oriented toward interior of $V$ (see fig.1). 
Thus the inequality above implies the following. 
%%%%%%%%%%%%%%%%%%%%%%%%%%%%%%%%%
\begin{proposition}\label{inward}
If $V\in\PLL_+^n$, then all the facet normals $\BS{z}_p$ 
of $V$ are  the inward ones. 
\end{proposition}
%%%%%%%%%%%%%%%%%%%%%%%%
\begin{proposition}\label{loop2loop}
The maps $\varphi$ and $\psi$ preserve 
nondegeneracy and positivity. Hence, 
the maps  $\varphi$ and $\psi$ 
naturally induces 
${\varphi}_+:{\PLL}_+^n\longrightarrow{\PLL}_+^n$ and 
${\psi}_+:{\PLL}_+^n\longrightarrow{\PLL}_+^n$ 
respectively. 
(Note that positivity of $\psi(L)$ does not always imply 
positivity of $L$.) 
\end{proposition}
%%%%%%%%%%%%%%%%%%%%%%%%%%%%%%%%%%%
\begin{proof}
This is obvious from Proposition \ref{scalar}.  
%%%%%%%%%%%%%%%%%%%%%%%%%%%%%%%%%%%
%\renewcommand{\tablename}{\textsf{diagram}}
\unitlength=1mm
\begin{table}[h]
\begin{center}
\fbox{
\begin{picture}(116,16)
%%%%%%%%%%%%%%%%%%%%%%%%%%%%%%%%
\put(108,-4.5){\scriptsize\textsf fig. 2}
\thicklines
\put(11,11){$\PLL_+^n$}
\put(13,7){\rotatebox{90}{$\in$}}
\put(12.5,4){$V$}
\put(5,0){vertex loop}
\put(24,13.6){$\varphi_+$: bijective}
\path(20.5,12)(51,12)
\path(49.5,12.8)(51,12)(49.5,11.2)
\path(20.5,5)(51,5)\path(20.5,4)(20.5,6)
%\path(24,5)(47,5)\path(24,4)(24,6)
\path(49.5,6)(51,5)(49.5,4)
\put(53.5,11){$\PLL_+^n$}
\put(56,7){\rotatebox{90}{$\in$}}
\put(55.5,4){$W$}
\put(51,0){edge loop}
\path(62,12)(94,12)
\path(62,5)(94,5)\path(62,4)(62,6)
\path(92.5,12.8)(94,12)(92.5,11.2)
\path(92.5,5.8)(94,5)(92.5,4.2)
\put(67,13.6){$\psi_+$: bijective?}
\put(96.5,11){$\PLL_+^n$}
\put(99,7){\rotatebox{90}{$\in$}}
\put(99,4){$Z$}
\put(94,0){facet loop}
%%%%%%%%%%%%
\end{picture}
}
\end{center}
\end{table}
\end{proof}
%%%%%%%%%%%%%%%%%
\begin{corollary}\label{bijective}
The map ${\psi}_+$ is bijective. 
\end{corollary}
%%%%%%%%%%%%%%%
\begin{proof}
Suppose that 
$(\psi_+\COMP\varphi_+)^2(V)=(\psi_+\COMP\varphi_+)^2(V')$ for 
$V,\ V'\in\PLL_+^n$. Then we have 
$(\det\UL{V})^{n-2}\cdot\UL{V}=(\det\UL{V}')^{n-2}\cdot\UL{V}'$ 
by the last part of the proof of Theorem \ref{iteration}. 
This implies $\det\UL{V}=\det\UL{V}'$ by positivity and 
$\UL{V}=\UL{V}'$ follows. 
This proves that $(\psi_+\COMP\varphi_+)^2$ is injective. 
If $V\in\PLL_+^n$, we have 
$$
(\psi_+\COMP\varphi_+)^2(\UL{t\cdot V})
=
(n+1)^{n-1}t^{n^2-2n+1}(\det\UL{V})^{n-2}\cdot\UL{V}
\qquad (t>0)
.$$
Choosing $t$ with $(n+1)^{n-1}t^{n^2-2n+1}(\det\UL{V})^{n-2}=1$, 
we see that $(\psi_+\COMP\varphi_+)^2$ is surjective. 
Since $\varphi_+$ is bijective as well as $\varphi$, $\psi_+$ is also so. 
\end{proof}
%%%%%%%%%%%%%%%%%%%
\begin{remark}
\begin{enumerate}
\item
If $n$ is odd and if $\det\UL{V}<0$, we see that 
$(\psi\COMP\varphi)^2(-V)=(\psi\COMP\varphi)^2(V)$. 
Hence $\psi\COMP\varphi$ is not injective in this case. 
\item
The $n$-simplices $V$ and $(\psi_+\COMP\varphi_+)^2(V)$ 
are in the position of similarity with centre at the origin O. 
\item
For a $2$-simplex $W$, $\psi(W)$ is 
the rotation with angle $\pi/2$. 
\end{enumerate}
\end{remark}
%%%%%%%%%%%%%%%
\section{On properties 
$\protect{\rm{(a),\ (b),}}$ and $\protect{\rm{(c)}}.$}
%%%%%%%%%%%%%%%%%%%%%%%%%%%%%%
Now let us reconfirm properties (a), (b) and (c) stated 
in Introduction in our context. 
Applying permutation of suffixes and translation, 
we may assume that simplices and loops have positive 
orientation and the simplices have centre at the origin O, 
namely, that they belong to $\PLL_+^n$. 
%%%%%%%%%%%%%%%%%%%%
\vskip1.5ex

\noindent
\textit{Proof of} (a). 
This is just the assertion that the images of $\psi_+^n$ 
are loops, 
which is shown in Proposition \ref{facetsformloop}. 
\hfill$\Box$
\vskip1.5ex

%%%%%%%%%%%%
\noindent
\textit{Proof of} (b). 
This is equivalent to say that $\psi_+^n$ is surjective, which 
is proved in Proposition \ref{bijective}. \hfill$\Box$
\vskip1.5ex

%%%%%%%%%%%%
\noindent
\textit{Proof of} (c). 
The inequalities 
$2|\BS{z}_k| \le |\BS{z}_0|+\cdots+|\BS{z}_n|\ (0\le k\le n)$ 
are trivial consequence of triangle inequality in 
Euclidean spaces. 
If the equality holds for some $k$, we have 
$$
|\BS{z}_0+\cdots+\check{\BS{z}}_k+\cdots+\BS{z}_n|
=
|\BS{z}_0|+\cdots+|\BS{z}_{k-1}|+|\BS{z}_{k+1}|
+\cdots+|\BS{z}_n|
.$$
Then we can prove that all $\BS{z}_p$ $(p\neq k)$ have 
the same direction by induction, contradicting the assumption 
that $\BS{z}_0,\dots,\BS{z}_n$ span $\BB{R}^n$. \hfill$\Box$
%%%%%%%%%%%%%%%%%%%%%%%%%%%%%%%%%%%%%
\section{Sufficiency of simplex inequalities.}
%%%%%%%%%%%%%%%%%%%%%%%%%%%%%%%%%%%
Now we prove our main assertion (d) in Introduction: 
\vskip1.5ex

\emph{
If $z_0,\dots,z_n > 0$ $(n\ge 2)$ 
satisfy the simplex inequalities, 
there exist vectors $\BS{z}_0,\dots,\BS{z}_n$ with 
$\BS{z}_0+\dots+\BS{z}_n=\BS{0}$ and 
$|\BS{z}_0|=z_0,\dots,|\BS{z}_n|=z_n$ 
which generate $\BB{R}^n$. 
}
%%%%%%%%%%%%%%%%%%%%%%%%%%%
\begin{lemma}
Take an orthonormal coordinate system $x_1,\dots,x_n$ 
of $\BB{R}^n$. 
Let $\BB{R}^k\subset\BB{R}^n$ denote the subspace defined by 
$x_{k+1}=\cdots= x_n=0$. 
Suppose that 
\begin{gather*}
0<z_0\le\dots\le z_n,\quad  2z_n < z_0+\dots+z_n,
\quad
0 < |\alpha_2|,\dots,|\alpha_{n-1}| \le \pi
%\quad \alpha_1,\dots,\alpha_n\neq 0
,\end{gather*} 
and ${\rm P}_n={\rm O}$, then there exist points 
${\rm P}_{k}\in\BB{R}^{k+1}\setminus\BB{R}^{k}$ 
$(0\le k\le n-1)$ such that the following holds. 
\begin{enumerate}
\item
$|{\rm P}_{k} - {\rm P}_{k-1}|=z_k\ 
(0\le k\le n,\ {\rm P}_{-1}:={\rm P}_n)$. 
\item
The angles between the triangles 
$\triangle{\rm P}_n{\rm P}_{k-1}{\rm P}_{k-2}$ and 
$\triangle{\rm P}_n{\rm P}_{k-1}{\rm P}_{k}$ are 
$\alpha_k$ $(2\le k\le n-1)$ 
%and $\angle{\rm P}_n{\rm P}_0{\rm P}_1=\alpha_1$ 
(choosing the positive directions of angles 
$\alpha_k$ arbitrarily). 
\end{enumerate}
\end{lemma}
%%%%%%%%%%%%%%%%%%%%%%%%%%%%
This lemma is sufficient to prove (d). 
For, we may assume that the sequence $\{z_k\}$ is 
non-decreasing by a permutation of vertices, and 
if we put $\BS{z}_k:=\ORA{{\rm P}_{k-1}{\rm P}}_k$, 
we have $\BS{z}_0+\dots+\BS{z}_n=\BS{0}$. 
%The matrix $\UL{Z}$ is triangular and its $(k,k)$-element  
%is $f_k\times\sin\alpha_k$, where $f_1=z_1$ and 
%$f_k$ $(2\le k\le n-1)$ is 
%the distance of ${\rm P}_{k}$ from the line 
%${\rm P}_{k-1}{\rm P}_n$, 
%a positive function of $(\BS{z}_0,\dots,\BS{z}_{k-1})$. 
Since ${\rm P}_{k}\in\BB{R}_{k+1}\setminus\BB{R}_{k}$ 
$(1\le k\le n-1)$, $\BS{z}_0,\dots,\BS{z}_n$ 
generate $\BB{R}^n$. 
If $\det(\UL{Z})<0$ ($Z:=(\BS{z}_0,\dots,\BS{z}_n)$), 
reverse the positive direction of 
$\alpha_{n-1}$ and then we have 
$Z=(\BS{z}_0,\dots,\BS{z}_n)\in\PLL_+^n$ with $|\BS{z}_k|=z_k$. 
Furthermore, since $\alpha_i$ assures degrees of freedom
$n-2$, this lemma implies that there are an infinite number of
essentially different solutions for $n\ge 3$. 
%%%%%%
\begin{proof} 
Let us prove by induction on $n$. 
The case $n=2$ is elementary (but not trivial). 
Suppose that we have proved till the case of $n-1$ $(n\ge 3)$. 

If $z_{n-2}<z_n$, we have 
$$
\max\{z_{n-2},z_n-z_{n-1}\}<\min\{z_0+\dots+z_{n-2},z_n\}
.$$ 
This assures the existence of $z_{n-1}'\in\BB{R}$ with
$$
z_{n-2}< z_{n-1}'< z_0+\dots+z_{n-2},
\quad z_n-z_{n-1}<z_{n-1}'< z_n
.$$
If $z_{n-2}=z_{n-1}=z_n$, 
we put $z_{n-1}':=z_{n-2}=z_{n-1}=z_n$. 
In both of these cases, we have 
%\begin{gather*}
\begin{gather*}
0 < z_0\le\dots\le z_{n-2}\le z_{n-1}',\quad
2z_{n-1}' < z_0+\dots+z_{n-2}+z_{n-1}',\\
z_n < z_{n-1}+z_{n-1}',\quad 
z_{n-1} < z_{n-1}'+z_n,\quad z_{n-1}'<z_{n-1}+z_n.
\end{gather*}
%\end{gather*}
By the first two inequalities, 
$z_0,\dots,z_{n-2}, z_{n-1}'$ 
satisfies the inductive assumption. 
Hence there exist points 
${\rm P}_0,{\rm P}_1,\dots,{\rm P}_{n-2},
{\rm P}'_{n-1}:={\rm O}$ such that 
\begin{gather*}
{\rm P}_{k}\in\BB{R}^{k+1}\setminus\BB{R}^{k}\ (1\le k\le n-2),
\\
|{\rm P}_{k} - {\rm P}_{k-1}|=z_k\ (0\le k\le n-2),\quad 
|{\rm P}'_{n-1} - {\rm P}_{n-2}|=z_{n-1}',
\end{gather*}
the angles between 
$\triangle{\rm P}'_{n-1}{\rm P}_{k-1}{\rm P}_{k-2}$ and 
$\triangle{\rm P}'_{n-1}{\rm P}_{k-1}{\rm P}_k$ are 
$\alpha_k$ $(2\le k\le n-2)$, where ${\rm P}_{-1}:={\rm O}$. 
%and $\angle{\rm P}_n{\rm P}_0{\rm P}_1=\alpha_1$. 
%%%%%%%%%%%%%%%%%%
%\caption{\textsf{\scriptsize dualities and Taylor projector}}
%\label{tab:1}    % Give a unique label
%%%%%%%%%%%%%%%%%%%%%%%%%%%%%%%%%%%
%\renewcommand{\tablename}{\textsf{diagram}}
\unitlength=1mm
\begin{table}[h]
\begin{center}
\begin{picture}(53,48)(-7,0)
%%%%%%%%%%%%%%%%%%%%%%%%%%%%%%%%
\put(22,0){\scriptsize\textsf fig. 3}
\dottedline{1}(5,30)(10,20)(21,10)(36,5)
\path(5,30)(36,5)
\path(36,5)(48,22)
\path(48,22)(5,30)
\path(5,30)(48,46)(48,22)
\HEAD{10}{20}{295}
\WHEAD{5.3}{29.8}{165}
\HEAD{21}{10}{316}
\HEAD{36}{5}{347}
\HEAD{48}{22}{60}
\HEAD{48}{46}{88}
\HEAD{5}{30}{200}
\put(2,24){$\BS{z}_{0}$}
\put(28,22){$\BS{z}_{n-1}'$}
\put(23,40){$\BS{z}_{n}$}
\put(49,34){$\BS{z}_{n-1}$}
\put(43,13){$\BS{z}_{n-2}$}
\put(-23,29){P${}'_{n-1}=\rm{O}\ (=\rm{P}{}_n)$}
\put(3,19){P$_0$}
\put(38,5){P$_{n-3}$}
\put(49,20){P$_{n-2}$}
\put(49,45){P${}_{n-1}$}
%%%%%%%%%%%%
\end{picture}
\end{center}
\end{table}
%%%%%%%%%%%%%%%%%%%%%%%%
Since $z_{n-1}$, $z_{n-1}'$, $z_n$ satisfies 
the triangle inequalities, 
there exists P${}_{n-1}\in\BB{R}^n\setminus\BB{R}^{n-1}$ 
such that 
$|{\rm P}_{n-1} - {\rm P}_{n-2}|=z_{n-1}$, 
$|{\rm P}'_{n-1} - {\rm P}_{n-1}|=z_n$ 
and the angle between 
$\triangle{\rm P}'_{n-1}{\rm P}_{n-2}{\rm P}_{n-1}$ and 
$\triangle{\rm P}'_{n-1}{\rm P}_{n-2}{\rm P}_{n-3}$ 
is $\alpha_{n-1}$. 
Renaming ${\rm P}'_{n-1}$ to ${\rm P}_n$, we have the desired 
points. 
\end{proof}
%%%%%%%%%%%%%%%%%%%%%%%%%
\subsection*{Acknowledgement.}
The author wishes to express thanks to Yoshio Agaoka for 
much bibliographical information. Furthermore he talked about 
Langr's problem, which gave the author a nice hint.  

\bibliographystyle{amsplain}

\end{document}